\documentclass[a4paper,10pt]{amsart}
\usepackage[utf8]{inputenc}

\usepackage{amsmath}
\usepackage{amstext}
\usepackage{amssymb}
\usepackage{amsthm}
\usepackage{enumerate}

    \DeclareMathOperator{\cov}{cov}
    \DeclareMathOperator{\OD}{OD}
    
    \DeclareMathOperator{\WO}{WO}

    \DeclareMathOperator{\dom}{dom}

 \renewcommand{\L}{L}

    \newcommand{\card}[1]{\vert #1\vert}
    \newcommand{\seq}[1]{\langle #1 \rangle }

\def\Fin{\operatorname{Fin}}

\usepackage{amsmath}
\usepackage{amssymb}
\usepackage{amsthm}

\usepackage{enumerate}

\theoremstyle{plain}
\newtheorem{thm}{Theorem}[section]
\newtheorem{lemma}[thm]{Lemma}

\newtheorem{cor}[thm]{Corollary}
\newtheorem*{claim}{Claim}
\newtheorem*{fact}{Fact}
\newtheorem*{remark}{Remark}
\theoremstyle{definition}
\newtheorem{definition}{Definition}[section]

\newtheorem{quest}{Question}[section]

\title{Coanalytic ultrafilter bases}

\author{Jonathan Schilhan}
\thanks{\emph{Acknowledgments:} The author would like to thank the Austrian Science Fund, FWF, for generous support through START-Project Y1012-N35.}

\begin{document}

\begin{abstract}
 We study the definability of ultrafilter bases on $\omega$ in the sense of descriptive set theory. As a main result we show that there is no coanalytic base for a Ramsey ultrafilter, while in $L$ we can construct $\Pi^1_1$ P-point and Q-point bases. We also show that the existence of a $\mathbf\Delta^1_{n+1}$ ultrafilter is equivalent to that of a $\mathbf\Pi^1_n$ ultrafilter base, for $n \in \omega$. Moreover we introduce a Borel version of the classical ultrafilter number and make some observations.
\end{abstract}
\maketitle

\section{Introduction}

This paper follows the line of many papers studying the definability, in the sense of descriptive set theory, of certain combinatorial subsets of the real line such as mad families \cite{Miller1989},\cite{Fischer2013},\cite{Raghavan2012},
independent families \cite{Miller1989},\cite{Brendle2019}, maximal eventually different families \cite{Schrittesser2018}, maximal cofinitary groups \cite{Fischer2017}, maximal orthogonal families of measures \cite{Fischer2012} or maximal towers and inextendible linearly ordered towers \cite{Fischer2018}. In this paper we will study the definability of ultrafilters and more specifically ultrafilter bases. Filters will always live on $\omega$ and contain all cofinite sets. Thus a filter is a subset of $\mathcal{P}(\omega)$ and we can study its definability. It is well known that an ultrafilter can neither have the Baire property nor be Lebesgue measurable.
This already rules out the existence of analytic ultrafilter generating sets as the generated filter will also be analytic and thus have the Baire property. But this still leaves open the possibility of a coanalytic ultrafilter base since a priori the generated set will only be $\mathbf\Delta^1_2$. Recall that for $x,y \in [\omega]^\omega$ we write $x \subseteq^* y$ whenever $x \setminus y$ is finite. An ultrafilter $\mathcal{U}$ is called a \emph{P-point} if for any countable $\mathcal{F} \subseteq \mathcal{U}$, there is $x \in \mathcal{U}$ so that $\forall y \in \mathcal{F} (x \subseteq^* y)$. $\mathcal{U}$ is a \emph{Q-point} if for any partition $\langle a_n : n \in \omega \rangle$ of $\omega$ into finite sets $a_n$, there is $x \in \mathcal{U}$ so that $\forall n \in \omega (\vert x \cap a_n\vert \leq 1)$. A \emph{Ramsey} ultrafilter is an ultrafilter that is both a P- and a Q-point. A more commonly known and equivalent definition for Ramsey ultrafilters $\mathcal{U}$ is that for any coloring $c \colon [\omega]^2 \to 2$, there is $x \in \mathcal{U}$ so that $c$ is homogeneous on $x$, i.e. $c \restriction [x]^2$ is constant. In fact we will show in Section 3 that:  

\begin{thm}\label{pi11ppoint}
 There is a $\Pi^1_1$ base for a P-point in the constructible universe $L$. 
\end{thm}

\begin{thm}\label{pi11qpoint}
 There is a $\Pi^1_1$ base for a Q-point in the constructible universe $L$. 
\end{thm}

Section 2 will provide an introduction to the techniques employed in proving these results. In strong contrast we will show in Section~\ref{sec:noramsey} that:

\begin{thm}\label{nopi11ramsey}
 There is no $\mathbf{\Pi}^1_1$ base for a Ramsey ultrafilter. 
\end{thm}

Notice that any ultrafilter that is $\mathbf\Sigma^1_n$ or $\mathbf\Pi^1_n$ is already $\mathbf\Delta^1_n$. Namely suppose that $\varphi$ defines an ultrafilter, then we have that $\varphi(x) \leftrightarrow \neg \varphi(\omega \setminus x)$. Moreover any base for an ultrafilter that is $\mathbf\Sigma^1_n$ or $\mathbf\Pi^1_n$ generates a $\mathbf\Delta^1_n$ or respectively a $\mathbf\Delta^1_{n+1}$ ultrafilter. 

In Section 5 we will compare $\mathbf\Delta^1_2$ ultrafilters to $\mathbf\Pi^1_1$ bases. As a main result we find that: 

\begin{thm}\label{thm:delta1n}
The following are equivalent for any $r \in 2^\omega$, $n \in \omega$. 

\begin{enumerate}
 \item There is a $\Delta^1_{n+1} (r)$ ultrafilter. 
 \item There is a $\Pi^1_n(r)$ ultrafilter base. 
\end{enumerate}
\end{thm}

In Section 6 we study the effects of adding reals to the definability of utrafilters. In Section 7 we introduce a new cardinal invariant that is a Borel version of the classical ultrafilter number $\mathfrak{u}$ and make some observations.

\section{An introduction to Miller's coding technique}

All our constructions of coanalytic objects in $\L$ will rely on a technique streamlined by A. Miller in his far-reaching paper \cite{Miller1989}. We outline here the theory behind this technique in a most general way. 

When we say that $z$ codes the ordinal $\alpha$, we mean the following. To any real $z \in 2^\omega$ we associate a relation $E_z$ on $\omega$ defined by 
$$E_z(n,m) \leftrightarrow z(2^n 3^m) =1.$$ This relation may be a linear order and if it is a well-order and isomorphic to $\alpha$ we say that it codes $\alpha$. Such $\alpha$ is unique and we define $\| z \| := \alpha$. More generally we say that $z$ codes $M$ if $(\omega,E_z)$ is isomorphic to $(M,\in)$. The set of $z \in 2^\omega$ coding an ordinal is denoted $\WO$. The set $\WO$ is tightly connected to coanalytic sets. On one hand side, $\WO$ is itself $\Pi^1_1$ and on the other, for any $\mathbf{\Pi}^1_1$ set $X \subseteq 2^\omega$, there is a continuous function $f \colon 2^\omega \to 2^\omega$ so that $X = f^{-1}(\WO)$.

There is a very canonical way of defining in $\L$ various combinatorial subsets $X$ of reals in a $\bf{\Delta^1_2}$ fashion. 
Typically the elements are found recursively by making adequate choices which are absolute between models of the form $L_\alpha$ (e.g. taking the $<_\L$ least candidate which has some simple property holding with respect to the previously chosen reals).

Then $x \in X$ can be written as

\begin{equation} \underbrace{\text{}\exists M}_{\exists} \text{}[\text{}\underbrace{ M\text{ is well-founded}}_{\forall}, \underbrace{\text{}x \in M}_{\Delta^1_1} \text{ and } \underbrace{M \models V = L \wedge \varphi(x)}_{\Delta^1_1}\text{}] \end{equation}

or as

\begin{equation} \underbrace{\text{}\forall M \models V=L, x \in M}_{\forall + \Delta^1_1} [\text{} \underbrace{ M\text{ is not well-founded }}_{\exists} \text{ or } \underbrace{M \models \varphi(x)}_{\Delta^1_1}\text{}]. \end{equation}

Quantifying over models is shorthand for quantifying over codes in $2^\omega$ of countable models satisfying some basic set theoretic axioms. Thus e.g. (1) can be recast as ``$\exists z \in 2^\omega ((\omega,E_z) \text{ is well-founded, } x \in (\omega,E_z) \text{ and } (\omega,E_z) \models V = L \wedge \varphi(x))$'', where $x \in (\omega, E_z)$ means that $x \in M$ for $M$ the Mostowski collapse of $(\omega,E_z)$. It is not difficult to see that this can be expressed in a $\Delta^1_1$ way. 

As such, finding a $\Delta^1_2$ ultrafilter base in $\L$ is very simple. The major improvement in Miller's technique is to get rid of the first existential quantifier in (1). 
This is done by letting $x$ already encode a relevant well-founded model $M$ in a Borel or even in a recursive way. Then if $C$ is the Borel coding relation used, the definition usually looks as follows: \\

\begin{equation} \underbrace{x \in Y}_{\forall} \text{ and }\underbrace{\text{}\forall z \in 2^\omega}_{\forall} [\underbrace{\neg C(x,z)}_{\Delta^1_1} \text{ or }  \underbrace{ (\omega,E_z) \models V=L \wedge \varphi(x)}_{\Delta^1_1}], \end{equation} for some known coanalytic $Y$.

\begin{lemma}\label{lem:Lalpha[x]isborel}
 There is a lightface Borel set $C \subseteq (2^\omega)^3$ so that whenever $z$ codes $\alpha < \omega_1$ and $r,y \in 2^\omega$ then $(z,r,y) \in C$ iff $y$ codes $L_\alpha[r]$.
\end{lemma}

\begin{proof}
The claim is easy to verify by noting that an adequate $E_{y}$ can be constructed by recursion on $\alpha$. Thus $(z,r,y) \in C$ can be defined by formulas of the form ``$\exists$/$\forall \seq{E_k : k \in \omega}$ a sequence indexed via the order coded by $z$ satisfying certain recursive assumptions, $E_{y}$ is the union of all $E_k$''. This definition is uniform on $z$ and $r$.
\end{proof}

\begin{lemma}\label{lem:plusomega}
There is a recursive function $(\cdot)^{+\omega} \colon 2^\omega \to 2^\omega$ so that whenever $z$ codes $\alpha$, then $(z)^{+\omega}$ codes $\alpha + \omega$. 
\end{lemma}

\begin{proof}
 Let $(z)^{+\omega} = y$ such that $y(2^n3^m) = 1$ iff $\begin{cases} n \text{ even } \wedge m \text{ even } \wedge z(2^\frac{n}{2}3^\frac{m}{2}) = 1 \\ n \text{ even } \wedge m \text{ odd } \\ n \text{ odd } \wedge m \text{ odd } \wedge n < m. \end{cases}$ 
\end{proof}

\section{$\Pi^1_1$ bases for P- and Q-points}

In \cite{Fischer2018} the authors constructed, using Miller's technique, a coanalytic tower (i.e. a set $X \subseteq [\omega]^\omega$ well-ordered wrt $^*\supseteq$ and with no pseudointersection). A crucial property of the tower was that all its elements were split by the set of even natural numbers. In particular this meant that the tower could not generate an ultrafilter. We will construct in $L$ a tower generating an ultrafilter and thus generating a P-point.

Before we start to construct the $\Pi^1_1$ P-point base, we need some ingredients. 

\begin{definition}
 We call $\mathcal{W}^+$ the set of $x \in [\omega]^\omega$ containing arbitrary long arithmetic progressions, i.e. $\forall k \in \omega \exists a,b \in \omega ( \{a\cdot l + b : l <k \} \subseteq x)$.
\end{definition}

The following fact follows from Van der Waerden's Theorem which is well known. 

\begin{fact}
 The set $\mathcal{W} = \mathcal{P}(\omega) \setminus \mathcal{W}^+$ is a proper ideal on $\omega$. It is called the Van der Waerden ideal.
\end{fact}

\begin{proof}[Proof of Theorem~\ref{pi11ppoint}]
Let $(y_\alpha)_{\alpha< \omega_1}$ enumerate $[\omega]^\omega$ via the global $L$ well-order $<_L$. The statement ``$y$ is the $\alpha$'th element according to $<_L$'' is absolute between $L_\beta$'s with $y \in L_\beta$ and $\alpha \in L_\beta$. Let $O \colon 2^\omega \to 2^\omega$ be the following lightface Borel function: If $x \subseteq \omega$ we want to define a unique sequence $(i_n)_{n \in \omega}$ of subsets of $\omega$ so that $\max i_n < \min i_{n+1}$ and $i_{n+1}$ is the next maximal arithmetic progression in $x$ of length $\geq 3$ above $\max i_{n}$ (note that any pair of natural numbers forms an arithmetic progression). Now if this sequence can be defined up to $\omega$ (in particular every $i_n$ is finite), then we define $O(x)(n) = 1$ iff $i_n$ has even length. Else we let $O(x)(n)=0$. 

 We construct a sequence $(x_\xi, \delta_\xi)_{\xi < \omega_1}$ where $x_\xi \in [\omega]^\omega$, $\delta_\xi < \omega_1$ as follows.
 
  Given $(x_\xi, \delta_\xi)_{\xi < \alpha}$ we let $\delta_\alpha$ be the least limit ordinal such that $ \sup_{\xi < \alpha}{\delta_\xi} < \delta_\alpha $, $y_\alpha \in L_{\delta_\alpha}$ and $\delta_\alpha$ projects to $\omega$, i.e. $L_{\delta_\alpha + \omega} \models \delta_\alpha \text{ is countable}$. It is not difficult to see that the set of ordinals projecting to $\omega$ is unbounded in $\omega_1$. $x_\alpha = x$ is chosen least in the $<_L$ well-order so that 
  \begin{enumerate}[(a)]
   \item $x \subseteq^* x_\xi$ for every $\xi < \alpha$, 
   \item $x \in \mathcal{W}^+$
   \item $x \subseteq y_\alpha$ or $x \subseteq \omega \setminus y_\alpha$.
   \item $O(x)$ codes $\delta_\alpha$. 
  \end{enumerate}

  Note that any sequence $(x_\xi)_{\xi < \omega_1}$ defined as above is a tower generating an ultrafilter. 
  
  \begin{claim}
   $x_\alpha$ can be found in $L_{\delta_\alpha+ \omega}$.
  \end{claim}

  \begin{proof}
  Note that the definition of $(x_\xi)_{\xi < \alpha}$ is absolute between $L_\beta$'s. In particular $(x_\xi)_{\xi < \alpha}$ can be defined over $L_{\delta_\alpha}$. As $\delta_\alpha$ projects to $\omega$, there is an enumeration $(x^n)_{n \in \omega}$ of $\{x_\xi:  \xi < \alpha\}$ in $L_{\delta_\alpha+ \omega}$. Given $y_\alpha$ we have that, as $\mathcal{W}$ is an ideal, that for every $\xi < \alpha$, $y_\alpha \cap x_\xi \in \mathcal{W}^+$ or $\omega \setminus y_\alpha \cap x_\xi \in \mathcal{W}^+$. Assume wlog that for cofinally many $x_\xi$, $y_\alpha \cap x_\xi \in \mathcal{W}^+$ is the case. This implies that for all $x_\xi$ this is the case as $(x_\xi)_{\xi < \alpha}$ forms a tower. Again as $\delta_\alpha$ projects to $\omega$, there is a real $z \in L_{\delta_\alpha+ \omega} \cap 2^\omega$ coding $\delta_\alpha$. Now we define a sequence $(i_n)_{n\in \omega}$ of finite subsets of $\omega$ so that $\max i_n < \min i_{n+1}$, $i_n \subseteq y_\alpha \cap \bigcap_{k \leq n} x^k$, $i_n$ consists of an arithmetic progression so that its length is $\geq n$ and it is even iff $z(n) = 1$. Moreover $\min i_n$ is chosen large enough so that $i_{n-1} \cup i_n$ cannot form an arithmetic progression. $x := \bigcup_{n \in \omega} i_n$ can be defined in $L_{\delta_\alpha+ \omega}$ and satisfies (a)-(d). Thus in particular the $<_L$-least such $x$ exists in $L_{\delta_\alpha + \omega}$.  
  \end{proof}
  
 \begin{remark}
  There is a formula $\varphi(x)$ in the language of set theory so that $\varphi(x)$ iff $\exists \xi (x = x_\xi)$ and $L_{\beta} \models \varphi(x)$ for some $\beta$ implies that $\varphi(x)$ is true. Moreover $L_{\delta_\xi + \omega} \models \varphi(x_\xi)$ for every $\xi$. 
 \end{remark}

 \begin{proof}
 $\varphi(x)$ expresses that there is an ordinal $\alpha$ and a sequence $(x_\xi, \delta_\xi)_{\xi \leq \alpha}$ according to the recursive definitions given above so that $x = x_\alpha$.
 \end{proof}

 Now we can check that the set $X = \{x_\xi : \xi \in \omega_1 \}$ is $\Pi^1_1$. Let $C$ and $(\cdot)^{+\omega}$ be as in Lemma~\ref{lem:Lalpha[x]isborel} and Lemma~\ref{lem:plusomega}. 
  Then $x \in X$ iff $$ O(x) \in \WO \text{ and }\forall z [\neg C(O(x)^{+\omega},0,z) \text{ or } (\omega,E_z) \models \varphi(x)] .$$
\end{proof}

\begin{definition}
 The ideal $\Fin^2$ on $\omega \times \omega$ consists of $x \in \mathcal{P}(\omega \times \omega)$ so that $\forall^\infty n \in \omega \forall^\infty m \in \omega (\langle n,m \rangle \notin x)$
\end{definition}

\begin{proof}[Proof of Theorem~\ref{pi11qpoint}]
 The ultrafilter that we construct will live on $\omega \times \omega$. Let $O \colon (\Fin^2)^+ \to 2^\omega$ be the following Borel function. Given $x \in (\Fin^2)^+$ let $x_0,x_1$ be the first two infinite vertical sections of $x$. We denote with $x_0(n)$ or $x_1(n)$ the $n$'th element of $x_0$ or $x_1$. Then $$O(x)(n) = \begin{cases}
                                                                                                                                                                                                                                                                                                                                                                                                                                                                                                  0 \text{ if } x_0(n)\geq x_1(n) \\
                                                                                                                                                                                                                                                                                                                                                                                                                                                                                                  1 \text{ if } x_1(n) > x_0(n).
                                                                                                                                                                                                                                                                                                                                                                                                                                                                                            \end{cases}
$$

As in the proof of Theorem~\ref{pi11ppoint} we let $(y_\alpha)_{\alpha < \omega_1}$ enumerate $[\omega \times \omega]^\omega$  and $(P_\alpha)_{\alpha < \omega_1}$ enumerate all partitions of $\omega \times \omega$ into finite sets via the well-ordering $<_L$. 

Similarly to the proof of Theorem~\ref{pi11ppoint} we construct a sequence $(x_\xi, \delta_\xi)_{\xi < \omega_1}$ where $x_\xi \in (\Fin^2)^+$, intersections of finitely many elements in $\{x_\xi : \xi < \omega_1 \}$ are in $(\Fin^2)^+$ and $\delta_\xi < \omega_1$ as follows.
 
  Given $(x_\xi, \delta_\xi)_{\xi < \alpha}$ we let $\delta_\alpha$ be the least limit ordinal such that $ \sup_{\xi < \alpha}{\delta_\xi} < \delta_\alpha $, $y_\alpha, P_\alpha \in L_{\delta_\alpha}$ and $\delta_\alpha$ projects to $\omega$, i.e. $L_{\delta_\alpha + \omega} \models \delta_\alpha \text{ is countable}$. $x_\alpha = x$ is then chosen least in the $<_L$ well-order so that 
  \begin{enumerate}[(a)]
   \item $\{ x \} \cup \{x_\xi : \xi < \alpha \} $ has all finite intersections in $(\Fin^2)^+$, 
   \item $x \in (\Fin^2)^+$, 
   \item $x \subseteq y_\alpha$ or $x \subseteq \omega \setminus y_\alpha$, 
   \item for every $a \in P_\alpha$, $\vert a \cap x \vert \leq 1$, 
   \item $O(x)$ codes $\delta_\alpha$. 
  \end{enumerate}

  Again we show that such an $x_\alpha$ exists and can be found in $L_{\delta_\alpha + \omega}$. 
  
  \begin{claim}
   $x_\alpha$ can be found in $L_{\delta_\alpha + \omega}$. 
  \end{claim}

  \begin{proof}
  We have that if $(x_\xi)_{\xi<\alpha}$ exists then it must be definable over $L_{\delta_\alpha}$. As $\delta_\alpha$ projects to $\omega$ there is in $L_{\delta_\alpha + \omega}$ an enumeration $(x^n)_{n\in\omega}$ of all finite intersections of elements in $\{x_\xi : \xi <\alpha \}$. We are given $y_\alpha \in L_{\delta_\alpha}$. It is not hard to see that either $y_\alpha$ or $(\omega \times \omega) \setminus y_\alpha$ is in $(\Fin^2)^+$ and has $(\Fin^2)^+$ intersection with all $x^n$. Without loss of generality we assume $y_\alpha$ has this property. Let $P_\alpha = \{ a_i : i \in \omega \}$ and $z \in 2^\omega \cap L_{\delta_\alpha + \omega}$ code $\delta_\alpha$. Further let $k_0<k_1$ be first so that the $k_0$'th and $k_1$'th vertical section of $y_\alpha$ is infinite.  Let $(p_j)_{j \in \omega}$ enumerate $\omega \times \omega$ in a way that every pair $(n,m)$ appears infinitely often.                                                                                                                                                                                                                                                                                                                                                                                                                                                                                                                                                                                                                                                                                                                                                                                                                                                                                                                                                                                                                                                                                                                                                                                                                                                                                                                                                                                                                                                                                                                                                                                                                                                                                                         Given $(p_j)_{j \in \omega}$ we define recursively a sequence $\langle m^0_i , m^1_i \rangle_{i \in \omega}$ and auxiliarily $(n_i)_{i \in \omega}$ as follows: 

\begin{itemize}
 \item for every $i$, $\langle m^0_i , m^1_i \rangle \in y_\alpha$, $\langle m^0_i , m^1_i \rangle \notin \bigcup_{j< i} a_{n_j}$ and $\langle m^0_i , m^1_i \rangle \in a_{n_i}$,
 \item if $i = 3 j$ for $j \in \omega$, then $\langle m^0_i , m^1_i \rangle$ is in the $p_j(0)$'th infinite vertical section of $y_\alpha \cap x^{p_j(1)}$ greater than $k_1$,
 \item if $i = 1 \mod 3$ then $m^0_i = k_0$ and $m^0_{i+1} = k_1$ and $m^1_i \geq m^1_{i+1}$ or $ m^1_{i+1} > m^1_i$ depending on whether $z(i) =0$ or $z(i) =1$.  
\end{itemize}

Now the set $\{\langle m_i^0,m_i^1 \rangle : i \in \omega \} \in L_{\delta_\alpha +\omega}$ satisfies (a)-(e) as can be seen from the construction. In particular $L_{\delta_\alpha + \omega}$ contains the $<_L$-least such set.                                                                                                                                                                                                                                                                                                                                                                                                                                                                                                                                                                                                                                                                                                                                                                                                                                                           
\end{proof}

The set $\{x_\xi : \xi < \omega_1 \}$ is now a base for a Q-Point and as in the proof of Theorem~\ref{pi11ppoint} it is $\Pi^1_1$. 
\end{proof}

\section{There are no $\mathbf\Pi^1_1$ Ramsey ultrafilter bases}

\label{sec:noramsey}

\begin{definition}
 Let $\mathcal{F}$ be a filter. Then the forcing $\mathbb{M}(\mathcal{F})$ consists of pairs $(a,F) \in [\omega]^{<\omega} \times \mathcal{F}$ such that $\max a < \min F$. A condition $(b,E)$ extends $(a,F)$ if $b$ is an end-extension of $a$, $E \subseteq F$ and $b \setminus a \subseteq F$.   
\end{definition}

$\mathbb{M}(\mathcal{F})$ is the natural forcing to add a pseudointersection of $\mathcal{F}$.

\begin{definition}
 Let $\mathcal{F}$ be a filter. Then we define the game $G(\mathcal{F})$ as follows: \\
 
 \begin{tabular}[h]{c | c c c c c c}
  Player I & $F_0 \in \mathcal{F}$ & & $F_1 \in \mathcal{F}$ & & \dots \\ 
  Player II & & $a_0 \in [F_0]^{<\omega} \setminus \{\emptyset \}$ & & $a_0 \in [F_1]^{<\omega} \setminus \{\emptyset \}$ & & \dots 
  
 \end{tabular} \\

Player II wins iff $\bigcup_{n \in \omega} a_n \in \mathcal{F}$. 
\end{definition}

\begin{lemma}
\label{lem:game}
 Let $\mathcal{F}$ be a filter on $\omega$. Then TFAE: 
 \begin{enumerate}[(i)]
  \item For any countable model $M$, $\mathcal{F} \in M$, of enough set theory, there is $x \in \mathcal{F}$, $\mathbb{M}(\mathcal{F})$ generic over $M$. 
  \item I has no winning strategy in $G(\mathcal{F})$. 
 \end{enumerate}
\end{lemma}

\begin{proof}
 (i) implies (ii): Suppose $\sigma$ is a winning strategy for I in $G(\mathcal{F})$ and let $\sigma, \mathcal{F} \in M$. Wlog we assume that $\sigma (\langle \rangle) = \omega$. Thus Player II is allowed to play any $a_0$ as his first move and then $\sigma$ carries on as if $a_0$ had not been played. In particular this means that any initial play $a_0$ of II is a legal move, i.e. $\langle a_0 \rangle \in \operatorname{dom}(\sigma)$.
 Consider the dense sets $D_n := \{ (s,F) : F \subseteq \bigcap \{ \sigma (\langle s_0,\dots,s_{n-1} \rangle) : \langle s_0,\dots,s_{n-1} \rangle \in \operatorname{dom}( \sigma), \bigcup_{i< k} s_i = s\} \}$ for $n \in \omega$. 
 $D_n \in M$ for every $n \in \omega$. By (i) there is $x \in \mathcal{F}$, $\mathbb{M}(\mathcal{F})$ generic over $M$. This means that for every $n \in \omega$ there is $s$ an initial segement of $x$ and $F \in \mathcal{F}$ so that $(s,F) \in D_n$ and $x \setminus s \subseteq F$.
 Now using this construct a sequence $\langle s_i \rangle_{i \in \omega}$ and $\langle F_i \rangle_{i \in \omega}$ recursively so that: 
 
 \begin{enumerate}
  \item $\bigcup_{i < n} s_i$ is an initial segment of $x$ for every $n \in \omega$,
  \item $\max s_{i} < \min s_{i+1} $ for every $i \in \omega$,
  \item $x \setminus \bigcup_{i < n} s_i \subseteq F_n$ for every $n \in \omega$,
  \item $(\bigcup_{i < n} s_i, F_n) \in D_{n}$.
 \end{enumerate}

 We find recursively that $\langle s_i \rangle_{i < n} \in \operatorname{dom}(\sigma)$, i.e. $\langle s_i \rangle_{i < n}$ is a legal move. But $\bigcup_{i \in \omega} s_i = x \in \mathcal{F}$ contradicting $\sigma$ being a winning strategy for I. 

 (ii) implies (i): Let $M \ni \mathcal{F}$ be countable and $\langle D_n \rangle_{n \in \omega}$ enumerate all dense subsets of $\mathbb{M}(\mathcal{F})$ in $M$. We describe a strategy for Player I: I starts by playing some $F_0$ so that there is $(t_0, F_0) \in D_0$. Then Player II will play $a_0 \subseteq F_0$, i.e. $(t_0 \cup a_0, F_0) \leq (t_0,F_0)$. Now I plays $F_1$ so that there is $(t_0 \cup a_0 \cup t_1, F_1) \in D_1$, $(t_0 \cup a_0 \cup t_1, F_1)\leq (t_0 \cup a_0, F_0)$...
 
 By assumption there is a winning run $\langle a_i \rangle_{i \in \omega}$ for II according to this strategy. This means that $\bigcup a_i \in \mathcal{F}$ and moreover $x = \bigcup a_i \cup \bigcup t_i \in \mathcal{F}$ where $t_i$ are as described. But $x$ is now $\mathbb{M}(\mathcal{F})$ generic over $M$. 
 \end{proof}

It is a well known theorem that for ultrafilters $\mathcal{U}$, I not having a winning strategy in $G(\mathcal{U})$ is equivalent to $\mathcal{U}$ being a P-point. For sake of completeness we prove a more general (in light of Lemma~\ref{lem:game}) version of this below. Recall that $\mathfrak{p}$ is the pseudointersection number, i.e. the least size of a set $\mathcal{B} \subseteq [\omega]^\omega$ with the finite intersection property and no pseudointersection, a set $x \in [\omega]^\omega$ such that $x \subseteq^* y$ for all $y \in \mathcal{B}$. The bounding number $\mathfrak{b}$ is the least size of a family $\mathcal{B} \subseteq \omega^\omega$ such that there is no $f \in \omega^\omega$ eventually dominating every member of $\mathcal{B}$. It is well known that $\aleph_1 \leq \mathfrak{p} \leq \mathfrak{b}$. An ultrafilter $\mathcal{U}$ is called a $P_\kappa$ point if for any $\mathcal{B} \in [\mathcal{U}]^{<\kappa}$ there is a pseudointersection $x \in \mathcal{U}$ of $\mathcal{B}$. In particular a $P$-point is the same as a $P_{\aleph_1}$-point.

\begin{lemma}
 Assume $\kappa \leq \mathfrak{p}$ and $\mathcal{U}$ is an ultrafilter. Then TFAE: 
 \begin{enumerate}[(i)]
  \item $\mathcal{U}$ is a $P_{\kappa}$-point.
  \item For every $M$ a model of enough set theory with $\card{M} < \kappa$ and $\mathcal{U} \in M$, there is $x \in \mathcal{U}$ which is $\mathbb{M}(\mathcal{U})$ generic over $M$. 
 \end{enumerate}

\end{lemma}

\begin{proof}
 (ii) implies (i) is trivial. 
 
 (i) implies (ii): Let $\card{M} < \kappa \leq \mathfrak{p}$. Then as $\mathcal{U}$ is a P$_\kappa$-point, there is $U \in \mathcal{U}$ so that $U \subseteq^* V$ for every $V \in M \cap \mathcal{U}$. Define for every $D \in M$, which is a dense open subset of $\mathbb{M}(\mathcal{U})$ and every $V \in M \cap \mathcal{U}$ a function $f_{D,V} \colon \omega \to \omega$ so that for $n \in \omega$: $$\forall a \subseteq n \exists b \subseteq [n,f_{D,V}(n)) \exists V' \in M \cap \mathcal{U} ((a \cup b, V')\leq (a,V) \wedge (a \cup b, V') \in D \wedge U \setminus f_{D,V}(n) \subseteq V'). $$
 The set of functions $f_{D,V}$ is smaller than $\kappa \leq \mathfrak{p} \leq \mathfrak{b}$. Thus there is one $f \in \omega^\omega$ dominating all of them. Let $i_0 = 0$, $i_{n+1} = f(i_n)$. We write $I_n = [i_n,i_{n+1})$. As $\mathcal{U}$ is an ultrafilter, either $U_0 = \bigcup_{n \in \omega} I_{2n} \cap U$ or $U_1 = \bigcup_{n \in \omega} I_{2n +1} \cap U$ is in $\mathcal{U}$. Assume wlog that $U_0 \in \mathcal{U}$. 
 
 We define a $\sigma$-centered partial order $\mathbb{P}$ as follows. $\mathbb{P}$ consists of pairs $(s,F)$ where \begin{enumerate}
                                                                                                                    \item $s \colon n \to [\omega]^{<\omega}$ for some $n \in \omega$, 
                                                                                                                    \item $s(i) \subseteq I_i$ for every $i < n$, 
                                                                                                                                                                                                                                  \item $s(i) = U \cap I_i$ when $i$ is even,                                                                                                                                                                                                                             
                                                                                                                                                                                                                                \item $F \in \mathcal{U} \cap M$. 
                                                                                                                                                                                                                                 \end{enumerate}
                                                                                                                   
                                                                                                                   A condition $(t,F)$ extends $(s,E)$ iff $t \supseteq s$, $F \subseteq E$ and $(t(i) \subseteq E)$ whenever $i \in \dom t \setminus \dom s$ is odd. For any $D \in M$ which is dense in $\mathbb{M}(\mathcal{U})$ we define a subset of $\mathbb{P}$, $\tilde{D}$ as follows: 
                                                                                                                   $$ \tilde{D} = \{ (t,F) : (\bigcup_{i \in \dom t} t(i), F)\in D \}.$$

                                                                                                                   We claim that $\tilde D$ is dense in $\mathbb{P}$. Let $(s,E) \in \mathbb{P}$ be arbitrary. Then as $f_{D,E} <^* f $ there is $n \in \omega$ so that $[i_{2n+1},f_{D,E}(i_{2n+1})) \subseteq [i_{2n+1},i_{2n+2})$ and $2n+1 \geq \dom s$. Now extend $s$ to $s_0$ so that $\dom s_0 = 2n+1$ and $s_0(i) = \emptyset$ for $i \in  2n+1 \setminus  \dom s$ odd and $s_0(i) = U \cap I_i$ for $i$ even. By definition of $f_{D,E}$ there is $b \subseteq I_{2n+1}$ so that $\exists F \subseteq E$ with $(a \cup b, F) \in D$ where $a = \bigcup_{i<2n+1} s_0(i)$, $(a \cup b, F) \leq (a,E)$ and $U \setminus i_{2n+2} \subseteq F$. Let $t = s_0 \cup \{(2n+1,b) \}$. Then $(t,F) \leq (s,E)$ in $\mathbb{P}$ and $(t,F) \in \tilde{D}$. 
                                                                                                                   
                                                                                                                   Now as $\kappa \leq \mathfrak{p}$ and by Bell's theorem (see \cite{Blass2010}) there is a $\mathbb{P}$ generic real $g \colon \omega \to [\omega]^{<\omega}$ over $M$. But then $x := \bigcup_{i \in \omega} g(i) \in \mathcal{U}$ as $U_0 \subseteq x$ and $x$ is $\mathbb{M}(\mathcal{U})$ generic over $M$.  
\end{proof}

\begin{cor}
\label{cor:ppointgeneric}
 Suppose $\mathcal{U}$ is a P-point, $M$ countable and $\mathcal{U} \in M$. Then there is $x \in \mathcal{U}$, $\mathbb{M}(\mathcal{U})$ generic over $M$.
\end{cor}

\begin{lemma}[see {\cite[Chapter 24]{Halbeisen2012}}]
\label{lem:ramseygeneric}
 Assume $\mathcal{U} \in M$ is a Ramsey ultrafilter and $x$ is $\mathbb{M}(\mathcal{U})$ generic over $M$. Then every $y \subseteq^* x$ is $\mathbb{M}(\mathcal{U})$ generic over $M$. 
\end{lemma}

\begin{proof}[Proof of Theorem~\ref{nopi11ramsey}]
 Suppose $\mathcal{U}$ is a Ramsey ultrafilter with a coanalytic base $X \subseteq [\omega]^\omega$. As $X$ is coanalytic, there is a continuous function $f : 2^\omega \to 2^\omega$ so that $$x \in X \leftrightarrow f(x) \in \WO.$$ Let $M$ be a countable model elementary in some $H(\theta)$ where $\theta$ is large enough and $\mathcal{U}, f \in M$. As $\mathcal{U}$ is a P-point and by Corollary~\ref{cor:ppointgeneric}, there is $x \in \mathcal{U}$ that is $\mathbb{M}(\mathcal{U})$ generic over $M$. Moreover as $\mathcal{U}$ is Ramsey and by Lemma~\ref{lem:ramseygeneric}, any $y \subseteq^* x$ is also generic over $M$. Let $\alpha = M \cap \omega_1$ and let $y \in X$ be arbitrary such that $y \subseteq^* x$. 
 Let $\beta = \|f(y)\|$, then $\beta \in M[y]$. Thus $\beta < \alpha = M[y] \cap \omega_1$. As $y$ was arbitrary, we have shown that the set $X'= \{ y : f(y) \in WO \wedge \| f(y)\| \leq \alpha \} \subseteq X$ contains $\{ y \subseteq^* x : y \in X  \}$. This means that $X'$ also generates $\mathcal{U}$. But $X'$ is Borel and cannot generate an ultrafilter. 
\end{proof}

\section{$\mathbf\Delta^1_2$ versus $\mathbf\Pi^1_1$}

Using a result of Shelah we can show the following. 

\begin{thm}
It is consistent that every P-point is $\mathbf{\Delta}^1_2$ and has no $\mathbf{\Pi^1_1}$ base. 
\end{thm}

\begin{proof}
 This follows immediately by \cite[Theorem XVIII.4.1]{Shelah1998} and the subsequent remark, which states that starting from $L$ we can choose any Ramsey ultrafilter $\mathcal{U}$ and pass to an extension in which $\mathcal{U}$ generates the unique P-point up to permutation of $\omega$. Moreover this ultrafilter will stay Ramsey. 
 
 Thus let $\mathcal{U}$ be any (definition of a) $\Delta^1_2$ Ramsey ultrafilter in $L$. Now apply Shelah's theorem to this ultrafilter and pass to an extension $V$ of $L$ in which $\mathcal{U}^L$ generates the unique P-point and is Ramsey. In $V$, $\mathcal{U}^V$ will still have the finite intersection property and $\mathcal{U}^L \subseteq \mathcal{U}^V$ by Shoenfield-absolutness. Thus in $V$, $\mathcal{U}^V$ generates the same ultrafilter as $\mathcal{U}^L$. As $\mathcal{U}^V$ is $\Delta^1_2$ the ultrafilter it generates will be $\Delta^1_2$ as well. We know that in $V$ there is for every P-point $\mathcal{V}$ a permutation $f$ of $\omega$ so that $V \in \mathcal{V} \leftrightarrow f(V) \in \mathcal{U}$. In particular $\mathcal{V}$ has a $\Delta^1_2(f)$ definition. On the other hand, every P-point is a Ramsey ultrafilter so none of them can have a $\mathbf{\Pi}^1_1$ base by Theorem~\ref{nopi11ramsey}.   
 
\end{proof}

\begin{proof}[Proof of Theorem~\ref{thm:delta1n}]
 To simplify notation we assume that $r=0$. Let $\mathcal{U}$ be a $\Delta^1_{n+1}$ ultrafilter. Let us introduce the following notation. For $y \in [\omega \times \omega]^\omega$, we let $y_n$ be $y$'s $n$'th vertical section. We let $z(y) = \{ n \in \omega : y_n \neq \emptyset \}$. When $z(y)$ is infinite then we denote with $y^n$, the $n$'th nonempty vertical section of $y$.  
 
 The Fubini product of $\mathcal{U}$, $\mathcal{U} \otimes \mathcal{U}$, consists of all $y \in [\omega \times \omega]^\omega$ so that $$\{n \in \omega : y_n \in \mathcal{U} \} \in \mathcal{U} .$$
 
 $\mathcal{U} \otimes \mathcal{U}$ is again an ultrafilter. We will show that it has a $\Pi^1_n$ base. Let $\varphi(x,w)$ be $\Pi^1_1$ so that $$x \in \mathcal{U} \leftrightarrow \exists w \in 2^\omega (\varphi(x,w)).$$

 Let $r : \omega \times 2^\omega \to 2^\omega$ be a recursive function such that for any sequence $\langle w_n \rangle_{ n \in \omega}$ there is $w \in 2^\omega$, which is not eventually constant, so that $r(n,w) = w_n$ for every $n \in \omega$. 
 
 Let $O \colon [\omega \times \omega]^\omega \to 2^\omega$ be the function defined by $$ O(y)(n) = \begin{cases} 0 \text{ if } \vert z(y) \vert < \omega \\
		0 \text{ if }  \min y^n \geq \min y^{n+1}\\
		1 \text{ if } \min y^n < \min y^{n+1}.
                                                                                          
                                                                                         \end{cases}
$$

$O$ is obviously lightface Borel. Let us define $X \subseteq [\omega \times \omega]^\omega$ as follows: $$y \in X \leftrightarrow \vert z(y) \vert = \omega \wedge \varphi(z(y),r(0,O(y))) \wedge \forall n \in \omega \exists s \in [\omega]^{<\omega} [\varphi(s \cup y^n,r(n+1,O(y)))].$$

$X$ is obviously $\Pi^1_n$. Moreover $X \subseteq \mathcal{U} \otimes \mathcal{U}$. To see this let us decode what $y \in X$ means. The first clause in the definition of $X$ says that $y$ has infinitely many nonempty vertical sections. The next clause ensures that $z(y) \in \mathcal{U}$ as witnessed by $r(0,O(y))$, the $0$'th real coded by $O(y)$. The last clause ensures that for every nonempty vertical section $y^n$ of $y$, $s \cup y^n$ is in $\mathcal{U}$ for some finite $s$ as witnessed by $r(n+1,O(y))$, the $n+1$'th real coded by $O(y)$. In particular $y^n \in \mathcal{U}$. Thus we indeed have that $y \in X \rightarrow y \in \mathcal{U} \otimes \mathcal{U}$.  

Moreover we have that $X$ is a base for $\mathcal{U} \otimes \mathcal{U}$. To see this fix $u \in \mathcal{U} \otimes \mathcal{U}$ and we show that there is $y \in X$ so that $y \subseteq u$. First let $y_0 = \bigcup \{ \{n \} \times u_n : n \in \omega, u_n \in \mathcal{U} \}$, i.e. we remove from $u$ the vertical sections that are not in $\mathcal{U}$. Then we let $w_0$ be such that $\varphi(z(y_0), w_0)$ holds true. Further we let $w_{n+1}$ be such that $\varphi(y_0^n, w_{n+1})$ holds true. Let $w \in 2^\omega$ be a single real coding the sequence $\langle w_n\rangle_{n \in \omega}$ via $r$, i.e. $r(n,w) = w_n$ for every $n \in \omega$. Find a sequence $\langle m_n \rangle_{n \in \omega}$ so that $m_n \in y_0^n$ for every $n$ and $w(n) = 1$ iff $m_{n+1} > m_n$. Such a sequence can be constructed recursively. Whenever $w(n) = 1$ we can simply find $m_{n+1} \in y_0^{n+1}$ large enough such that $m_{n+1} > m_n$ and if additionally $w(n+1), \dots, w(n+k)$ is a maximal block of $0$s in $w$ then we let $m_{n+1} = \dots = m_{n+k+1} \in y^{n+1} \cap \dots \cap y^{n+k+1}$. 
Finally given the sequence $\langle m_n \rangle_{n\in \omega}$ let $y = \bigcup \{\{z(y_0)(n)\} \times (y_0^n \setminus m_n)  : n \in \omega \}$, where $z(y_0)(n)$ is the $n$'th element of $z(y_0)$. We see that $y \subseteq y_0 \subseteq u$, that $z(y) = z(y_0)$, that $y^n =^* y_0^n$ for every $n$ and that $O(y) = w$. In particular $y \in X$ by definition of $X$. 

\end{proof}

\section{Adding reals}

Let $A \subseteq V$. A set $X \in V$ is called $\OD(A)$ if it is definable over $V$ from ordinals and elements of $A$ as parameters. Recall that a poset $\mathbb{P}$ is weakly homogeneous if for any $p,q \in \mathbb{P}$, there is an automorphism $\pi \colon \mathbb{P} \to \mathbb{P}$ so that $\pi(p)$ is compatible to $q$. In this section we will denote with $\mathcal{P}_A$ the collection of weakly homogeneous $\OD(A)$ posets. 

\begin{thm}
\label{thm:cohen}
 Let $c$ be a Cohen real over $V$, $\mathbb{P}\in (\mathcal{P}_V)^{V[c]}$ and $G$ a $\mathbb{P}$-generic filter over $V[c]$. Then in $V[c][G]$, $c$ is splitting over any set of reals with the finite intersection property that is $\OD(V)$. 
\end{thm}

\begin{proof}
 Let $X \in V[c][G]$ be an $\OD(V)$ set of reals with the finite intersection property, say $V[c][G] \models \text{``}\dot X = \{ x \in [\omega]^\omega : \varphi(x,a,\bar\alpha) \}\text{''}$ where $a \in V$ and $\bar\alpha$ is a finite sequence of ordinals. Wlog we may assume that $X$ is a filter, since the filter generated by $X$ is also $\OD(V)$. Suppose $c$ does not split $X$. This means exactly that $c \in X$ or $\omega \setminus c \in X$. Thus there is  $s \subseteq c$, deciding the formula and parameters defining $\mathbb{P}$, and $\dot p$ with $\dot p [c] \in G$, $(s,\dot p) \Vdash \text{``}\varphi\text{ defines a filter''}$ so that either $$(s,\dot p) \Vdash \dot c \in \dot X $$ or $$(s,\dot p) \Vdash  \omega \setminus \dot c \in \dot X.$$
 
 But now notice that $c' = s \cup \{ (n,1 - m) : (n,m) \in c, n \geq \vert s \vert \}$ is also Cohen over $V$ with $s \subseteq c'$ (we identify $c$ as a subset of $\omega$ with its characteristic function). Moreover $V[c] = V[c']$ and thus $\dot{\mathbb{P}}[c]  = \dot{\mathbb{P}}[c']$. Let $p_0 := \dot p[c]$ and $p_1 :=\dot p[c']$. Working in $V[c]$ we find that $p_0,p_1 \in \mathbb{P}$, so there is an automorphism $\pi$ of $\mathbb{P}$ so that $\pi(p_1)$ is compatible to $p_0$. Let $H$ be $\mathbb{P}$-generic over $V[c]$ containing $p_0$ and $\pi(p_1)$. In either of the above cases, $V[c][H]\models \varphi(c,a,\bar\alpha)\wedge \varphi(c',a,\bar\alpha)$. This is a contradiction to $(s,\dot p) \Vdash \text{``}\varphi\text{ defines a filter''}$. 
\end{proof}

\begin{thm}
\label{thm:random}
 Let $r$ be a random real over $V$, $\mathbb{P}\in (\mathcal{P}_V)^{V[r]}$ and $G$ a $\mathbb{P}$-generic filter over $V[r]$. Then in $V[r][G]$, $r$ is splitting over any set of reals with the finite intersection property that is $\OD(V)$. 
\end{thm}

\begin{proof}
 Let us assume that $\mathbb{P}$ is simply the trivial forcing, since this part of the argument is essentially the same as in the last proof. As before we fix $X \in V[r]$ an $\OD(V)$ set with the finite intersection property and we assume that it is already a filter. 
 
 First note that any finite modification of $r$ is still a random real. Moreover, as complementation is a measure preserving homeomorphism of $2^\omega$, the complement of a random real is still random. Thus any $r' =^* \omega \setminus r$ is still random. 
 
 Now similarly as in the proof for Cohen forcing we find that there is Borel set $B$ of positive measure coded in $V$ so that $r \in B$ and $$B \Vdash \dot r \in X$$ or $$B \Vdash \omega \setminus \dot r \in X.$$ 
 
 Recall that for any Borel set $A$ of positive measure, its $E_0$ closure $\tilde A = \{ x \in 2^\omega : \exists y \in A ( x =^* y) \}$ has full measure. To see this Let $\varepsilon >0$ be arbitrarily small. Apply Lebesgue's density theorem to find a basic open set $[s] \subseteq 2^\omega$ so that $\frac{\mu(A \cap [s])}{\mu([s])} > 1-\varepsilon$. Follow from this that $\mu (\tilde A) > 1- \varepsilon $. 
 
 Now let $C := \{ \omega \setminus x : x \in \tilde B \}$. $C$ is coded in $V$ and has full measure. Thus we have that $r \in B \cap C$. By definition of $C$, there is $r' \in B$ so that $r' =^* \omega \setminus r$. Moreover $r'$ is also a random real over $V$ by our first remark. $r,r' \in X$ and $\omega \setminus r, \omega \setminus r' \in X$ are both contradictions to $X$ having the finite intersection property. 
\end{proof}

Recall that Silver forcing consists of partial functions $p \colon \omega \to 2$ so that $\omega \setminus \dom(p)$ is infinite.

\begin{thm}
\label{thm:silver}
 Let $s$ be a Silver real over $V$, $\mathbb{P}\in (\mathcal{P}_V)^{V[s]}$ and $G$ a $\mathbb{P}$-generic filter over $V[s]$. Then, in $V[s]$, there is a real splitting over any set of reals that is $\OD(V)$ in $V[s][G]$. 
\end{thm}

\begin{proof}

Again we only consider the case when $\mathbb{P}$ is trivial. Let $X \in V[s]$ be an $\OD(V)$ filter. Let $S_s = \{n \in \omega : \vert \{ m < n : s(m) = 1 \} \vert \text{ is even} \}$. As before assume $p \subseteq s$ is such that either $$p \Vdash  S_{\dot s} \in X$$ or $$p \Vdash  \omega \setminus S_{\dot s} \in X.$$

Let $n = \min (\omega \setminus \dom(p))$ and note that $s'$ defined by $s'(i) = s(i)$ for all $i \neq n$ and $s'(n) = 1- s(n)$ is also Silver and $p \subseteq s'$. But $S_{s'} =^* \omega \setminus S_{s}$. We get the same contradiction as in the last two proofs. 
\end{proof}

\begin{cor}
 Let $r \in 2^\omega$ and assume that there is a Cohen, a random or a Silver real over $L[r]$. Then there is no $\Delta^1_2(r)$ ultrafilter.
 
 In particular, the existence of a $\Delta^1_2(r)$ ultrafilter implies that $\omega_1 = \omega_1^{L[r]}$.
\end{cor}

\begin{proof}
 Suppose that $\varphi$ is a $\Sigma^1_2(r)$ definition for an ultrafilter and that $c$ is a Cohen, random or Silver real over $L[r]$. In $L[r][c]$, the set defined by $\varphi$ will have the finite intersection property by downwards absoluteness. Thus by Theorem~\ref{thm:cohen}, \ref{thm:random} or \ref{thm:silver} respectively, $L[r][c] \models \exists x \in [\omega]^\omega \forall y \in [\omega]^\omega (\neg \varphi(y) \vee (\vert x \cap y \vert = \omega \wedge \vert x \cap \omega \setminus y \vert = \omega))$. This is a $\Sigma^1_3(x,c)$ statement, so by upwards Shoenfield absoluteness it holds true in $V \supseteq L[x][c]$. Thus $\varphi$ cannot define an ultrafilter in $V$. 
 
 The second part follows, since whenever $\omega_1^{L[r]} < \omega_1$, there is a Cohen real in $V$ over $L[r]$.  
\end{proof}

Another way of seeing the above for Cohen or random forcing is to use the classical result of Judah and Shelah (see \cite{Ihoda1989}), saying that the existence of a Cohen or random real over $L[r]$ is equivalent to every $\Delta^1_2(r)$ set having the Baire property or being Lebesgue measurable respectively.

\begin{cor}
 There is no $OD(\mathbb{R})$ ultrafilter, in particular no projective one, after adding $\omega_1$ many Cohen reals in a finite support iteration, random reals using a product of Lebesgue measure or Silver reals in a countable support iteration. 
\end{cor}

\begin{proof}
 Let $\langle c_\alpha : \alpha < \omega_1 \rangle$ be Cohen reals added via a finite support iteration over a ground model $V$ and suppose that in $V[\langle c_\alpha : \alpha < \omega_1 \rangle]$ there is an ultrafilter $\mathcal{U}$ definable from a real $a$ and ordinals. It is well known that there is $\xi < \omega_1$ so that $a \in V[\langle c_\alpha : \alpha \in \omega_1 \setminus \{\xi \} \rangle]$. But then, by Theorem~\ref{thm:cohen}, $c_\xi$ is splitting over $\mathcal{U}$, since $V[\langle c_\alpha : \alpha < \omega_1 \rangle] = V[\langle c_\alpha : \alpha \in \omega_1 \setminus \{\xi \} \rangle][c_\xi]$. 
 
 The argument for random reals is essentially the same. 
 
 Let $\langle \mathbb{P}_\alpha,\dot{ \mathbb{Q}}_\alpha : \alpha \leq \omega_1 \rangle$ be the $\omega_1$-length countable support iteration of Silver forcing. Any real $a$ appears in $V^{\mathbb{P}_\xi}$ for some $\xi < \omega_1$. But now note that $\mathbb{P}_{\omega_1}$ is $\OD(V)$ and weakly homogeneous. Moreover, $\mathbb{P}_{\omega_1} \cong \mathbb{P}_{\xi} * \dot{\mathbb{P}}_{\omega_1}$. Thus applying Theorem~\ref{thm:silver}, we find that there is no ultrafilter definable from parameters in $V^{\mathbb{P}_\xi}$ over $V^{\mathbb{P}_{\omega_1}}$. In particular there is no $\OD(\{a\})$ ultrafilter in $V^{\mathbb{P}_{\omega_1}}$.  
\end{proof}

\section{The Borel ultrafilter number}

The ultrafilter number $\mathfrak{u}$ is the least size of a base for an ultrafilter. As with mad families (see \cite{Raghavan2012}) and maximal independent families (see \cite{Brendle2019}) it makes sense to introduce a Borel version of the ultrafilter number that is closely related to the definability of ultrafilters. 

\begin{definition}
 The Borel ultrafilter number is defined as $$\mathfrak{u}_B := \min\{ \vert \mathcal{B} \vert : \mathcal{B} \subseteq \mathbf{\Delta}^1_1, \bigcup \mathcal{B} \text{ is an ultrafilter} \}.$$
\end{definition}

Note that $\aleph_1 \leq \mathfrak{u}_B$, as a countable union of Borel sets is Borel.

\begin{remark}
 Let $\mathfrak{u}'_B = \min\{ \vert \mathcal{B} \vert : \mathcal{B} \subseteq \mathbf{\Delta}^1_1, \bigcup \mathcal{B} \text{ is an ultrafilter base} \}$ and $\mathfrak{u}''_B = \min\{ \vert \mathcal{B} \vert : \mathcal{B} \subseteq \mathbf{\Delta}^1_1, \bigcup \mathcal{B} \text{ generates an ultrafilter} \}$. Then $\mathfrak{u}''_B =\mathfrak{u}'_B = \mathfrak{u}_B$.
\end{remark}

\begin{proof}
 Obviously, $\mathfrak{u}''_B \leq \mathfrak{u}'_B \leq \mathfrak{u}_{B}$. Remember that whenever $B$ is Borel, then the filter $F_B$ that it generates is analytic. Thus $\mathfrak{u}''_B$ is uncountable as well. Now let $\mathcal{B}$ be a collection of Borel sets, whose union generates an ultrafilter. We may assume that $\mathcal{B}$ is closed under finite unions. For every $B \in \mathcal{B}$, let $F_B$ be the filter generated by $B$. Since $F_B$ is analytic, we can write it as an $\omega_1$-union $F_B = \bigcup_{\alpha < \omega_1} F_B^\alpha$ of Borel sets. Now consider $\{ F_B^\alpha : B \in \mathcal{B}, \alpha < \omega_1 \}$. It has the same size as $\mathcal{B}$ and is a witness for $\mathfrak{u}_B$. 
\end{proof}

Any coanalytic set is an $\omega_1$-union of Borel sets. Thus the existence of a coanalytic ultrafilter base implies that $\mathfrak{u}_B = \aleph_1$.

\begin{thm}
 $\cov (\mathcal{M}),\cov (\mathcal{N}), \mathfrak{b}\leq \mathfrak{u}_B \leq \mathfrak{u}$.
\end{thm}

\begin{proof}
 Let $\mathcal{B}$ be a collection of $< \cov(\mathcal{M}) $ many Borel sets and assume that $\bigcup \mathcal{B}$ has the finite intersection property. Let $M \preccurlyeq H(\theta)$ for some large $\theta$, so that $\vert M \vert < \cov(\mathcal{M})$ and $\mathcal{B} \subseteq M$. Then there is a Cohen real $c$ over $M$. But then in $M[c]$, $c$ is splitting over every $B \in \mathcal{B}$. Moreover in $V$ it is true that $c$ is splitting over $B$, by $\Sigma_1$-upwards-absoluteness. Thus $c$ is splitting over $\bigcup \mathcal{B}$ which cannot be an ultrafilter. 
 The argument for random forcing is exactly the same. 
 
 For $\mathfrak{b} \leq \mathfrak{u}_B$, note that any Borel filter is meager. By a classical result of Talagrand (see \cite{Talagrand1980}), meager filters $\mathcal{F}$ are exactly those for which there is $f \in \omega^\omega$ so that $\forall x \in \mathcal{F} \forall^\infty n \in \omega (x \cap [n,f(n)) \neq \emptyset)$. For $\mathcal{B}$ a collection of Borel filters, we let $f_B$ be such a function for every $B \in \mathcal{B}$. If $\mathcal{B}$ has size smaller than $\mathfrak{b}$, then there is a single function $f \in \omega^\omega$ so that $f_B <^* f$ for each $B \in \mathcal{B}$. Now note that $x_0 \cup x_1 = \omega$, where $x_0 := \bigcup_{n \in \omega} [f^{2n}(0),f^{2n+1}(0))$ and $x_1 := \bigcup_{n \in \omega} [f^{2n+1}(0),f^{2n+2}(0))$. But neither $x_0$ nor $x_1$ can be in $\bigcup \mathcal{B}$. 
\end{proof}

\begin{quest}
 Is it consistent that $\mathfrak{u}_B < \mathfrak{u}$? Is it consistent that there is a $\Pi^1_1$ ultrafilter base while $\aleph_1 < \mathfrak{u}$? 
\end{quest}

\end{document}